\documentclass[a4paper,reqno]{amsart}
\usepackage[utf8]{inputenc}
\usepackage{amsmath,amssymb,amsfonts}
\usepackage[T1]{fontenc}

\usepackage[colorlinks=true, linkcolor=blue, citecolor=red] {hyperref}
\usepackage{enumitem}
\usepackage{graphicx}
\usepackage{units}
\allowdisplaybreaks

\newtheorem{theorem}{Theorem}
\newtheorem{lemma}[theorem]{Lemma}

\newtheorem{corollary}[theorem]{Corollary}

\newtheorem{definition}[theorem]{Definition}
\newtheorem{remark}[theorem]{Remark}

\numberwithin{equation}{section}
\numberwithin{theorem}{section}

\newcommand{\eps}{\varepsilon}

\newcommand{\R}{\mathbb{R}}
\newcommand{\N}{\mathbb{N}}

\DeclareMathOperator{\sign}{sign}
\DeclareMathOperator{\supp}{supp}

\newcommand{\Dx}{{\Delta x}}
\newcommand{\Dy}{{\Delta y}}
\newcommand{\DxDy}{{\mathbf{\Delta}}}

\newcommand{\Z}{\mathbb{Z}}
\renewcommand{\phi}{\varphi}
\renewcommand{\leq}{\leqslant}
\renewcommand{\geq}{\geqslant}
\newcommand{\dd}{\,d}
\newcommand{\cell}{\mathcal{C}}
\newcommand{\recon}{\mathcal{R}}
\newcommand{\hf}{{\unitfrac{1}{2}}}

\newcommand{\iphf}{{i+\hf}}
\newcommand{\imhf}{{i-\hf}}
\newcommand{\jphf}{{j+\hf}}
\newcommand{\jmhf}{{j-\hf}}
\newcommand{\jmpr}[1]{\langle\hspace*{-0.25em}\langle #1 \rangle\hspace*{-0.25em}\rangle}
\newcommand{\jmp}[1]{[\hspace*{-0.15em}[#1]\hspace*{-0.15em}]}
\newcommand{\loc}{{\text{loc}}}
\newcommand{\ip}[2]{\bigl\langle #1,\,#2\bigr\rangle}
\newcommand{\avg}[1]{\{\hspace*{-0.3em}\{#1\}\hspace*{-0.3em}\}}

\begin{document}
\title[Convergence of second order methods]{Convergence of second-order, entropy stable methods for multi-dimensional conservation laws}

\author[N. Chatterjee]{Neelabja Chatterjee}
\address[Neelabja Chatterjee]
{\newline Department of mathematics
\newline University of Oslo
\newline P.O. Box 1053,  Blindern
\newline N--0316 Oslo, Norway} 
\email[]{ \newline{nelabja12@gmail.com} \\ \newline{ neelabjc@math.uio.no}}

\author[U. S. Fjordholm]{Ulrik Skre Fjordholm}
\address[Ulrik Skre Fjordholm]
{\newline Department of mathematics
\newline University of Oslo
\newline P.O. Box 1053,  Blindern
\newline N--0316 Oslo, Norway} 
\email[]{ulriksf@math.uio.no}

\date{\today}


\thanks{This project has received funding from the European Union's
Horizon 2020 research and innovation programme under the Marie Sk\l{}odowska-Curie grant
agreement No 642768 for NC's work.}

\begin{abstract}
High-order accurate, \emph{entropy stable} numerical methods for hyperbolic conservation laws have attracted much interest over the last decade, but only a few rigorous convergence results are available, particularly in multiple space dimensions. In this paper we show how the entropy stability of one such method yields a (weak) bound on oscillations, and using compensated compactness we prove convergence to a weak solution satisfying at least one entropy condition.
\end{abstract}

\maketitle


\section{Introduction}
{Hyperbolic conservation laws} appear in a large variety of applications, including gas dynamics, traffic modeling, multi-phase fluid flow problems, and more; see e.g.~\cite{DafermosI,GodlewskiI,HR,LeVequeI}. We consider a scalar, $d$-dimensional hyperbolic conservation law
\begin{equation}\label{eq:SCLI}
\begin{cases}
\partial_tu + \nabla \cdot f(u) = 0 &\forall\ (x,t) \in \R^d \times \R_+ \\
u(x,0) = u_{0}(x) &\forall\ x \in \R^d
\end{cases}
\end{equation}
where $u = u(x,t): \R^d \times \R_+ \to \textbf{U}$ is the unknown \textit{conserved variable}, taking values in some nonempty interval $\mathbf{U}\subset\R$, and the function $f = (f_1, \dots , f_d):\textbf{U} \to \R^d$ is the smooth (at least $C^2$ on \textbf{U}) and possibly nonlinear \textit{flux function}. It is well-known that even if the initial datum $u_{0}(x)$ is arbitrarily smooth, the solutions of \eqref{eq:SCLI} may still be non-smooth \cite{DafermosI,GodlewskiI,HR}. Thus, it is fruitless to look for solutions of \eqref{eq:SCLI} in the classical sense. Instead these solutions are sought in a weak sense. A function $u \in L^\infty(\R^d \times \R_+)$ is said to be a \textit{weak solution} of \eqref{eq:SCLI} if it is a distributional solution, i.e.
\begin{equation}\label{eq:WeakFormulation}
\int_{\R^d} \int_{\R_+} \partial_tu \phi + f(u) \cdot  \nabla\phi \dd x \dd t + \int_{\R^d} u_{0}(x) \phi(x,0) \dd x = 0 \qquad \forall\ \phi \in C_c^1(\R^d \times [0, \infty)).
\end{equation}
It is well known (see e.g.~\cite{DafermosI,GodlewskiI,HR}) that weak solutions may be non-unique. Thus to single out a physically relevant solution, the notion of weak solution has to be supplemented with an additional admissibility criterion, namely \textit{entropy conditions}. A pair of functions $\eta:\R\to\R,q:\R\to\R^d$ is an \emph{entropy pair} for \eqref{eq:SCLI} if $\eta$, the \textit{entropy function}, is convex and $q$, the \textit{entropy flux function}, satisfies $q'(u) = \eta'(u)f'(u)$. In particular, for every $k\in\R$ we have the well-known \emph{Kru\v{z}kov entropy pair} $(\eta_k, q_k)$ given by
\begin{equation}\label{eq:KruzkovEntropy}
\eta_k(u) := |u - k|, \qquad q_k(u) := \sign(u - k) (f(u) - f(k)),
\end{equation}
see \cite{Kru70}. Multiplying \eqref{eq:SCLI} by $\eta'(u)$ and using the chain rule we have the following \textit{entropy conservation identity} for smooth solutions of \eqref{eq:SCLI},
\begin{equation}\label{eq:ContiEntropyConserveIdentity}
\partial_t\eta(u) + \nabla \cdot q(u) = 0.
\end{equation}
Due to the possible {non}-smoothness of solutions of \eqref{eq:SCLI}, the above derivation cannot be rigorously justified for weak solutions. Instead, the \textit{entropy inequality}
\begin{equation}\label{eq:ContiEntropyInequality}
\partial_t\eta(u) + \nabla \cdot q(u) \leq 0
\end{equation}
is imposed. As was shown by Kru\v{z}kov \cite{Kru70}, this \emph{entropy condition} guarantees uniqueness and stability of solutions of \eqref{eq:SCLI}.

\subsection{Numerical methods for conservation laws}
The nonlinear nature of the equation \eqref{eq:SCLI} and the fact that its solutions are irregular, can make the construction and analysis of numerical methods for \eqref{eq:SCLI} challenging. We outline here some of the available literature on this subject.

In order to converge to a weak solution satisfying the entropy condition \eqref{eq:ContiEntropyInequality}, the numerical method needs to satisfy some discrete version of the entropy condition. Finite volume methods (to be discussed in subsection \ref{subsec:FVM} in detail) with this property are called \emph{entropy stable}. Harten, Hyman and Lax \cite{HartenIII} showed that all monotone schemes for scalar conservation laws are entropy stable with respect to any entropy pair $(\eta,q)$. Osher \cite{OsherI} generalized this to a (presumably) larger class of schemes, the so-called \emph{E-schemes}. Osher also showed that these E-schemes are at most first-order accurate. In his papers from $1984$ \cite{TadmorI} and from $1987$ \cite{TadmorII}, Tadmor laid a general framework for constructing entropy stable schemes by first constructing \emph{entropy conservative} methods -- schemes satisfying a discrete version of \eqref{eq:ContiEntropyConserveIdentity} -- and then adding numerical diffusion to obtain entropy stability. As he proved in \cite{TadmorII}, entropy conservative schemes are generally second-order accurate; even higher-order entropy conservative schemes were constructed by Lefloch, Mercier and Rohde in \cite{LeflochI}. However, the addition of numerical diffusion to any of these entropy conservative schemes, in the way suggested in \cite{TadmorI,TadmorII}, degrades the order of accuracy to 1.

By combining the high-order accurate entropy conservative schemes in \cite{TadmorII,LeflochI} with a judiciously chosen \emph{reconstruction method}, Fjordholm, Mishra and Tadmor \cite{UlrikII,UlrikIII} constructed entropy stable methods with an arbitrarily high order of accuracy, the so-called \emph{TECNO schemes}. By estimating the amount of entropy dissipated by the method (\textit{i.e.}, the amplitude of the left-hand side in \eqref{eq:ContiEntropyInequality}), the authors could derive \emph{a priori} weak regularity bounds on the numerical solution, and these bounds were sufficient to prove convergence of the method in the special case of $d=1$ space dimensions.

To the best of our knowledge there is no available proof of convergence of a high-order accurate entropy stable method for a multi-dimensional conservation law. The purpose of the present paper is to prove convergence for a particular case, namely the second-order TECNO scheme.

\section{Entropy stable numerical methods}

\subsection{Finite volume methods}\label{subsec:FVM}
For the sake of notational simplicity we are going to consider the scalar conservation law \eqref{eq:SCLI} in the particular case of $d=2$ space dimensions, although we emphasize that the results in this paper are equally valid for any number of spatial dimensions $d$. 

We write \eqref{eq:SCLI} in the case $d=2$ as
\begin{equation}\label{eq:SCLII}
\begin{cases}
\partial_tu + \partial_xf^x(u) + \partial_y f^y(u) = 0 & \forall\ (x,y,t) \in \R^2 \times \R_+ \\
u(x,y,0) = u_{0}(x,y) & \forall\ (x,y) \in \R^2.
\end{cases}
\end{equation}
Here and in the remainder we will denote the components of all vector-valued functions by $f=(f^x,f^y)$.

One of the most popular discretization frameworks is \textit{finite volume methods}. The spatial domain $\R^2$ is partitioned into rectangles of the form $\cell_{i,j} = [x_\imhf,x_\iphf) \times [y_{\jmhf},y_\jphf)$, where we for the sake of simplicity use uniform grid sizes $x_\iphf - x_\imhf \equiv \Dx$ and $y_\jphf - y_\jmhf \equiv \Dy$. We denote the midpoint values as $x_i = \frac{x_\imhf + x_\iphf}{2}$ and $y_j = \frac{y_\jmhf + y_\jphf}{2}$. For any quantity $(u_{i,j})_{i,j\in\Z}$ defined on this grid, we define the jump and average operators
\begin{align*}
\avg{u}_{\iphf,j} = \frac{u_{i,j} + u_{i+1,j}}{2} && \jmp{u}_{\iphf,j} := u_{i+1,j} - u_{i,j} \\
\avg{u}_{i,\jphf} = \frac{u_{i,j} + u_{i,j+1}}{2} && \jmp{u}_{i,\jphf} := u_{i,j+1} - u_{i,j}.
\end{align*}
We let $u_{i,j}(t)$ be an approximation of the average value of $u$ over the rectangles $\cell_{i,j}$, that is,
\[
u_{i,j}(t) \approx \frac{1}{\Dx\Dy} \int_{\cell_{i,j}}u(x,y,t) \,d(x,y).
\]
The initial data is sampled as $u_{i,j}(0) = \frac{1}{\Dx\Dy} \int_{\cell_{i,j}}u_{0}(x,y)\,d(x,y)$. A semi-discrete finite volume method for \eqref{eq:SCLII} can then be written in the generic form
\begin{equation}\label{eq:SCLFVMI}
\frac{d}{dt}u_{i,j}(t) + \frac{F^x_{\iphf,j} - F^x_{\imhf,j}}{\Dx} + \frac{F^y_{i,\jphf} - F^y_{i,\jmhf}}{\Dy} = 0,
\end{equation}
where the numerical flux function $F^x_{\iphf,j} = F^x(u_{i,j},u_{i+1,j})$ is computed from an approximate solution of the Riemann problem at the interface $\{(x_\iphf, y)\}_{y_\jmhf\leq y \leq y_\jphf}$ and $F^y_{i,\jphf}$ is computed analogously \cite{GodlewskiI,HR}. The computed solution generated by the scheme is given by $u^\DxDy(x,y,t) = \sum_{i,j} u_{i,j}(t) \chi_{\cell_{i,j}}(x,y)$, where $\DxDy=(\Dx,\Dy)$ and $\chi_{\cell}$ is characteristic function for the rectangle $\cell$. We say that the numerical flux function $F$ is \emph{consistent with $f$} if $F^x(u, u) = f^x(u)$ and $F^y(u, u) = f^y(u)$ for all $u \in \mathbf{U}$. We also say that a numerical flux $F$ is \emph{locally Lipschitz continuous} if $F^x,F^y$ are locally Lipschitz continuous in each argument.

\subsection{Entropy stable numerical methods}\label{sec:entropystableschemes}

In order for any limit $u = \lim_{\DxDy\to0}u^\DxDy$ to satisfy the entropy condition \eqref{eq:ContiEntropyInequality}, the numerical method \eqref{eq:SCLFVMI} must satisfy some discrete form of the entropy condition. In this section we briefly review the theory of so-called \emph{entropy conservative} and \emph{entropy stable} schemes, and we define the \emph{TECNO} schemes, which will be the subject of the rest of the paper.

\begin{definition}[Entropy conservative methods]\label{def:entrcons}
Let $(\eta,q)$ be an entropy pair. We say that the finite volume method \eqref{eq:SCLFVMI} is \emph{entropy conservative} (with respect to $(\eta,q)$) if computed solutions satisfy the discrete entropy equality
\begin{equation}\label{eq:SCLFVMentropyconsvI}
\frac{d}{dt}\eta(u_{i,j}) + \frac{Q^x_{\iphf,j} - Q^x_{\imhf,j}}{\Dx} + \frac{Q^y_{i,\jphf} - Q^y_{i,\jmhf}}{\Dy} = 0,
\end{equation}
where $Q^x_{\iphf,j} = Q^x(u_{i,j},u_{i+1,j})$ and $Q^y_{i,\jphf} = Q^y(u_{i,j},u_{i,j+1})$ are \emph{numerical entropy flux functions} satisfying $Q^x(u,u) = q^x(u)$ and $Q^y(u,u) = q^y(u)$ for all $u\in\mathbf{U}$.
\end{definition}

\begin{definition}[Entropy stable methods]\label{def:entrstab}
Let $(\eta,q)$ be an entropy pair. We say that the finite volume method \eqref{eq:SCLFVMI} is \emph{entropy stable} (with respect to $(\eta,q)$) if computed solutions satisfy the discrete entropy equality
\begin{equation}\label{eq:SCLFVMentropystableI}
\frac{d}{dt}\eta(u_{i,j}) + \frac{Q^x_{\iphf,j} - Q^x_{\imhf,j}}{\Dx} + \frac{Q^y_{i,\jphf} - Q^y_{i,\jmhf}}{\Dy} \leq 0,
\end{equation}
where $Q^x_{\iphf,j} = Q^x(u_{i,j},u_{i+1,j})$ and $Q^y_{i,\jphf} = Q^y(u_{i,j},u_{i,j+1})$ are \emph{numerical entropy flux functions} satisfying $Q^x(u,u) = q^x(u)$ and $Q^y(u,u) = q^y(u)$ for all $u\in\mathbf{U}$.
\end{definition}

For an entropy pair $(\eta,q)$ the mapping $u \mapsto \eta'(u)$ is of particular importance, and we denote this \emph{entropy variable} by $v=v(u):=\eta'(u)$. If $\eta$ is strictly convex, $\eta''(u) > 0$, then the map $u \mapsto v(u)$ is strictly monotone increasing and hence is invertible. This inverse will be denoted by $u(v)$. Thus, the mapping $u \mapsto v$ induces a change of variables, in terms of which we can pose the conservation law \eqref{eq:SCLFVMI} as
\begin{equation}
\partial_tu(v) + \nabla \cdot f(u(v)) = 0.
\end{equation}
We define also the \textit{entropy potential} $\psi:\mathbf{U}\to\R^d$ defined by $\psi(u) := v(u)f(u) - q(u)$, whose name is given by the fact that $\partial_v \psi(u(v)) = f(u(v))$. 

A general approach to designing entropy conservative/stable schemes is as follows. Multiplying both sides of \eqref{eq:SCLFVMI} by $v_{i,j}:=\eta'(u_{i,j})$ and using the chain rule we get 
\[
\frac{d}{dt}\eta(u_{i,j}) + v_{i,j}\frac{F^x_{\iphf,j} - F^x_{\imhf,j}}{\Dx} + v_{i,j} \frac{F^y_{i,\jphf} - F^y_{i,\jmhf}}{\Dy} = 0.
\]
Adding and subtracting terms yields
\begin{equation}\label{eq:SCLFVMentropy}
\begin{split}
\frac{d}{dt}\eta(u_{i,j}) + \frac{Q^x_{\iphf,j} - Q^x_{\imhf,j}}{\Dx} + \frac{Q^y_{i,\jphf} - Q^y_{i,\jmhf}}{\Dy} \\ 
= \frac{r^{x}_{\iphf,j} + r^{x}_{\imhf,j}}{2\Dx} + \frac{r^{y}_{i,\jphf} + r^{y}_{i,\jmhf}}{2\Dy}
\end{split}
\end{equation}
where
\begin{equation}\label{eq:entropyresidualflux}
\begin{split}
r^{x}_{\iphf,j} &= \jmp{v}_{\iphf,j} F^x_{\iphf,j} - \jmp{\psi^{x}}_{\iphf,j}, \\
r^{y}_{i,\jphf} &= \jmp{v}_{i,\jphf} F^y_{i,\jphf} - \jmp{\psi^{y}}_{i,\jphf} \\
Q_{\iphf,j}^x &= \avg{v}_{\iphf,j} F^x_{\iphf,j} - \avg{\psi^{x}}_{\iphf,j}, \\
Q_{i,\jphf}^y &= \avg{v}_{i,\jphf} F^y_{i,\jphf} - \avg{\psi^{y}}_{i,\jphf}.
\end{split}
\end{equation}
It is straightforward to see that $Q^x, Q^y$ are consistent with $q$ in the sense of Definitions \ref{def:entrcons} and \ref{def:entrstab}, as long as $F^x,F^y$ are consistent with $f$. Thus, if $F^x,F^y$ are chosen such that either $r\equiv0$ or $r\leq0$, then the scheme \eqref{eq:SCLFVMI} is entropy conservative/stable. In particular, if $F^x,F^y$ are of the form
\begin{equation}\label{eq:entrstabflux}
F^x_{\iphf,j}=\tilde{F}^x_{\iphf,j} - D^{x}_{\iphf,j}\jmp{v}_{\iphf,j}, \qquad F^y_{i,\jphf}=\tilde{F}^y_{i,\jphf} - D^{y}_{i,\jphf}\jmp{v}_{i,\jphf}
\end{equation}
for nonnegative coefficients $D^{x}, D^{y}\geq0$ and numerical fluxes $\tilde{F}^x, \tilde{F}^y$ satisfying
\begin{equation}\label{eq:EntropyConsTadmor}
\jmp{v}_{\iphf,j} \tilde{F}^x_{\iphf,j} = \jmp{\psi^x}_{\iphf,j}, \qquad \jmp{v}_{i,\jphf} \tilde{F}^y_{i,\jphf} = \jmp{\psi^y}_{i,\jphf}
\end{equation}
then the resulting scheme \eqref{eq:SCLFVMI} is entropy stable. These observations were first made by Tadmor \cite{TadmorI,TadmorII}; see also \cite{TadmorIII}. For fluxes of the form \eqref{eq:entrstabflux} we also get a precise expression for the amount of entropy dissipated in \eqref{eq:SCLFVMentropystableI}:
\begin{align*}
\frac{d}{dt}\eta(u_{i,j}) + \frac{Q^x_{\iphf,j} - Q^x_{\imhf,j}}{\Dx} + \frac{Q^y_{i,\jphf} - Q^y_{i,\jmhf}}{\Dy} \\ 
= \frac{D^{x}_{\iphf,j}\jmp{v}_{\iphf,j}^2+D^{x}_{\imhf,j}\jmp{v}_{\imhf,j}^2}{2\Dx} + \frac{D^{y}_{i,\jphf}\jmp{v}_{i,\jphf}^2+D^{y}_{i,\jmhf}\jmp{v}_{i,\jmhf}^2}{2\Dy}.
\end{align*}
Under further assumptions on $\eta$ and $D$, this yields explicit bounds on ``weak TV'' terms of the form $\sum_{i,j}\jmp{v}_{\iphf,j}^2\Dy$, which can be used to prove compactness and convergence of the numerical method; see e.g.~\cite{UlrikIII}. We will apply this approach in Section \ref{sec:convergence}.

\begin{remark}
The above observations can be used to \emph{design} entropy stable schemes, by first finding numerical fluxes $\tilde{F}^x,\tilde{F}^y$ satisfying \eqref{eq:EntropyConsTadmor}, and then adding diffusion in the form \eqref{eq:entrstabflux}. We note that with this approach, we are only guaranteed that the discrete entropy inequality \eqref{eq:SCLFVMentropystableI} (or \eqref{eq:SCLFVMentropyconsvI} for entropy conservative schemes) is satisfied for \emph{one particular} entropy pair $(\eta,q)$.
\end{remark}

\subsection{The TECNO scheme}\label{subsec:TECNO}
The scheme \eqref{eq:SCLFVMI} with fluxes $\tilde{F}^x,\tilde{F}^y$ satisfying \eqref{eq:EntropyConsTadmor} is entropy conservative, in the sense of Definition \ref{def:entrcons}. It can be shown that two-point entropy conservative schemes are at most second-order accurate, in the sense of truncation error \cite{TadmorI, TadmorII}. When adding diffusion in the form \eqref{eq:entrstabflux} with $D=O(1)$, the resulting scheme is at most \emph{first-order} accurate. The TECNO schemes, introduced in \cite{UlrikIII,UlrikII}, represent a systematic approach to designing higher-order accurate entropy stable schemes. Since the convergence proof in Section \ref{sec:convergence} only applies to the second-order TECNO schemes, we will only describe these methods here, and we refer to \cite{UlrikIII,UlrikII} for the general construction.

The TECNO scheme has two main ingredients: An entropy conservative flux $\tilde{F}^x,\tilde{F}^y$, and a \emph{sign preserving reconstruction method}. Since our mesh is a Cartesian grid, we define the reconstruction procedure in a tensorial manner. For a partition $\{\cell_i\}_{i\in\Z}$ of $\R$ we consider a $p$th order reconstruction operator $\recon$, mapping any grid function $(w_i)_{i\in\Z}$ to a piecewise $(p-1)$th order polynomial $\recon w (x)$. Multi-dimensional grid functions $(w_{i,j})_{i,j\in\Z}$ are reconstructed dimension-by-dimension, defining in particular the edge values
\begin{subequations}\label{eq:Reconstruction}
\begin{equation}
\begin{split}
w^\pm_{\iphf,j} = \recon w_{\cdot,j}(x_\iphf\pm0), \qquad w^\pm_{i,\jphf} = \recon w_{i,\cdot}(y_\jphf\pm0)
\end{split}
\end{equation}
(where we by ``$+0$'' and ``$-0$'' mean right and left limits, respectively). We define also the edge jumps
\begin{equation}
\jmpr{w}_{\iphf,j} = w^+_{\iphf,j} - w^-_{\iphf,j}, \qquad \jmpr{w}_{i,\jphf} = w^+_{i,\jphf} - w^-_{i,\jphf}.
\end{equation}
\end{subequations}
Fix now some entropy pair $(\tilde{\eta},\tilde{q})$. The second-order TECNO scheme \cite{UlrikII,UlrikIII} is constructed from a flux $\tilde{F}^x,\tilde{F}^y$ which is entropy conservative with respect to $(\tilde{\eta},\tilde{q})$, and applies a second-order reconstruction method to the entropy variables $\tilde{v}=\tilde{\eta}\circ u$. The resulting scheme \eqref{eq:SCLFVMI} has numerical flux
\begin{equation}\label{eq:TECNOFlux}
\begin{split}
F^{x}_{\iphf,j} = \tilde{F}_{\iphf,j} - D^{x}_{\iphf,j} \jmpr{\tilde{v}}_{\iphf,j}, \\
F^{y}_{i,\jphf} = \tilde{F}_{i,\jphf} - D^{y}_{i,\jphf} \jmpr{\tilde{v}}_{i,\jphf},
\end{split}
\end{equation}
where $D^{x}, D^{y} \geq 0$. As shown in \cite{UlrikII}, the above scheme is formally second-order accurate, and it satisfies the discrete entropy inequality \eqref{eq:SCLFVMentropystableI} for the entropy pair $(\tilde{\eta},\tilde{q})$, provided the reconstruction operator $\recon$ satisfies the \emph{sign property}
\[
\jmp{\tilde{v}}_{\iphf,j}\jmpr{\tilde{v}}_{\iphf,j} \geq 0, \qquad \jmp{\tilde{v}}_{i,\jphf}\jmpr{\tilde{v}}_{i,\jphf} \geq 0.
\]
This is indeed true for the \emph{ENO reconstruction method} \cite{HartenI}:
\begin{theorem}[The ENO sign property \cite{UlrikIV}]\label{thm:ENOstability}
For some $p\in\N$, let $\recon$ denote the $p$-th order ENO reconstruction operator. Then for any grid function $(w_i)_{i\in\Z}$,
\begin{equation}\label{eq:ENOsignproperty}
\sign\jmpr{w}_\iphf = \sign\jmp{w}_\iphf.
\end{equation}
Moreover, there exists a constant $C_p>0$ depending only on $p$ such that
\begin{equation}\label{eq:ENOupperbound}
\frac{\jmpr{w}_\iphf}{\jmp{w}_\iphf} \leq C_p.
\end{equation}
\end{theorem}

For the sake of simplicity we henceforth select the entropy for which TECNO is entropy stable as $\tilde{\eta}(u)=u^2/2$. The corresponding entropy variable is then simply $\tilde{v}=u$, making the mapping between conserved and entropy variables somewhat easier. A summary of the TECNO scheme that we will analyze in this paper follows:

\begin{definition}\label{def:tecno}
The \emph{second-order TECNO scheme} for \eqref{eq:SCLII} is the numerical scheme \eqref{eq:SCLFVMI} with the numerical flux
\begin{equation}\label{eq:tecnoflux}
\begin{split}
F^{x}_{\iphf,j} = \tilde{F}_{\iphf,j} - D^{x}_{\iphf,j} \jmpr{u}_{\iphf,j}, \\
F^{y}_{i,\jphf} = \tilde{F}_{i,\jphf} - D^{y}_{i,\jphf} \jmpr{u}_{i,\jphf},
\end{split}
\end{equation}
where 
\begin{itemize}
\item $\tilde{F}$ is a consistent and locally Lipschitz continuous numerical flux which is entropy conservative with respect to the entropy $\tilde{\eta}(u)=u^2/2$
\item the diffusion coefficients $D$ satisfy
\[
\underline{D} \leq D^{x}_{\iphf,j},\, D^{y}_{i,\jphf} \leq \overline{D} \qquad \text{for fixed } \underline{D}, \overline{D}>0
\]
\item $\jmpr{u}_{\iphf,j}$ and $\jmpr{u}_{i,\jphf}$ denote the jumps in the second-order ENO reconstruction of the conserved variable $(u_{i,j}(t))_{i,j\in\Z}$.
\end{itemize}
\end{definition}

\begin{theorem}\label{thm:tecnoproperties}
The second-order TECNO scheme (cf.~Definition \ref{def:tecno}) has the following properties:
\begin{enumerate}[label=(\roman*)]
\item it is entropy stable with respect to the square entropy $\tilde{\eta}(u)=u^2/2$
\item the flux $F$ is locally Lipschitz continuous
\item there is some $C>0$ independent of $\Dx,\Dy$ such that
\begin{subequations}\label{eq:weakBVbound}
\begin{equation}\label{eq:weakBVboundu}
\int_0^T \sum_{i,j}\Bigl(\bigl|\jmp{u}_{\iphf,j}\bigr|^3\Dy + \bigl|\jmp{u}_{i,\jphf}\bigr|^3\Dx\Bigr) \,dt \leq C\|u\|_{L^\infty(\R^2\times[0,T])},
\end{equation}
\begin{equation}\label{eq:weakBVboundv}
\int_0^T \sum_{i,j} \Big(\jmp{u}_{\iphf,j}\jmpr{u}_{\iphf,j} \Dy + \jmp{u}_{i,\jphf}\jmpr{u}_{i,\jphf} \Dx \Big)\,dt \leq C.
\end{equation}
\end{subequations}
\end{enumerate}
\end{theorem}

\begin{proof}
The entropy stability follows from the calculations in Section \ref{sec:entropystableschemes} and the sign property \eqref{eq:ENOsignproperty}. Local Lipschitz continuity of $F$ follows from the Lipschitz continuity of $\tilde{F}$ and the upper bound \eqref{eq:ENOupperbound}.

For the ``weak TV bounds'' \eqref{eq:weakBVbound}, summing \eqref{eq:SCLFVMentropy} over $i,j\in\Z$, integrating over $t\in[0,T]$ and using the specific form of $F$ in \eqref{eq:tecnoflux} yields
\begin{align*}
\frac{1}{2}\sum_{i,j}(u_{i,j}(T))^2\,\Dx\Dy - \frac{1}{2}\sum_{i,j}(u_{i,j}(0))^2\,\Dx\Dy = -\mathcal{E}, \\
\mathcal{E} := \int_0^T\sum_{i,j}\Bigl(D^{x}_{\iphf,j}\jmp{u}_{\iphf,j}\jmpr{u}_{\iphf,j}\Dy + D^{y}_{i,\jphf}\jmp{u}_{i,\jphf}\jmpr{u}_{i,\jphf}\Dx\Bigr)\,dt.
\end{align*}
Since $D\geq0$ and the reconstruction satisfies the sign property \eqref{eq:ENOsignproperty}, we have $\mathcal{E}\geq0$. From the above we also see that $\mathcal{E} \leq \frac{1}{2}\sum_{i,j}(u_{i,j}(0))^2\,\Dx\Dy \leq \frac{1}{2}\|u_0\|_{L^2(\R^2)}^2$, so we see that the left-hand side of \eqref{eq:weakBVboundv} can be upper-bounded by $\underline{D}\mathcal{E} \leq \frac{1}{2}\underline{D}\|u_0\|_{L^2(\R^2)}^2 < \infty$.

For the remaining property \eqref{eq:weakBVboundu} we use the following fact, proved in \cite[Section 4.4]{Fjo16}: For every grid function $(w_i)_{i\in\Z}$ with compact support,
\begin{equation}\label{eq:enoconjecture}
\sum_{i \in \Z} \bigl|\jmp{w}_\iphf\bigr|^{3} \leq 2\|w\|_{l^\infty} \sum_{i \in \Z} \jmpr{w}_\iphf\jmp{w}_\iphf
\end{equation}
where $\jmpr{w}_\iphf=w_\iphf^+-w_\iphf^-$ denotes the jump in the second-order ENO reconstruction of $w$. Thus, the left-hand side of \eqref{eq:weakBVboundu} can be bounded by $\|u\|_{L^\infty}$ times the left-hand side of \eqref{eq:weakBVboundv}.
\end{proof}

\begin{remark}
For even higher-order TECNO schemes, the results in Theorem \ref{thm:tecnoproperties} are still valid, with the exception of \eqref{eq:weakBVboundu}: The crucial estimate \eqref{eq:enoconjecture} has been conjectured but remains unproven; cf.~ \cite[Section 4.4]{Fjo16} or \cite[Section 5.5]{UlrikIII}.
\end{remark}

\section{Convergence of the scheme}\label{sec:convergence}
Given a numerical solution $(u_{i,j}(t))_{i,j\in\Z, t\in\R_+}$ computed with the second-order TECNO scheme (cf.~Definition \ref{def:tecno}) we define the piecewise constant function
\[
u^\DxDy(x,y,t) := u_{i,j}(t) \qquad \text{for } (x,y)\in\cell_{i,j},
\]
where $\DxDy=(\Dx,\Dy)$. The goal of this section will be to show the following theorem:
\begin{theorem}\label{thm:convergence}
Assume that the solution $u^\DxDy$ computed by the TECNO scheme (cf.~Definition \ref{def:tecno}) is uniformly $L^\infty$ bounded, 
\begin{equation}\label{eq:linftybound}
\|u^\DxDy\|_{L^\infty(\R^2\times[0,T])} \leq M \qquad \text{for every } \DxDy=(\Dx,\Dy)>0
\end{equation}
for some $M>0$. Then there is some subsequence $\DxDy'\to0$ such that $u^{\DxDy'}\to u$ pointwise a.e.~and in $L^p(\R^2\times[0,T])$ for every $p\in[1,\infty)$. The function $u$ is a weak solution of \eqref{eq:SCLII} which satisfies the entropy condition \eqref{eq:ContiEntropyInequality} for the entropy $\eta(u)=u^2$.
\end{theorem}

We will use the method of compensated compactness, and we give the main results required here in Section \ref{sec:compcomp}. The convergence proof is given in Section \ref{sec:convergenceproof}, but we summarize it here:
\begin{proof}[Proof of Theorem \ref{thm:convergence}]
The compactness result, Corollary \ref{cor:PanovI} requires the entropy residuals $\{\partial_t\eta(u^\DxDy) + \nabla\cdot q(u^\DxDy)\}_{\DxDy>0}$ to lie in a compact subset of $H^{-1}_\loc$. Lemma \ref{lem:entropyresidualbound} bounds their discrete equivalents $\partial_t\eta(u^\DxDy) + \nabla\cdot Q(u^\DxDy)$ by terms which, by Theorem \ref{thm:tecnoproperties}\textit{(iii)} and \eqref{eq:linftybound}, are bounded in the sense of measures. Lemma \ref{lm:precompactness} shows that the remainder $\nabla\cdot(q-Q)$ is small in $H^{-1}_\loc$. We then conclude (using Murat's Lemma) that the entropy residuals are precompact, and hence there is some strongly convergent subsequence.

Lemma \ref{lem:LaxWendroff} is a standard ``Lax--Wendroff'' proof, showing that the limit is a weak solution, and Lemma \ref{lem:entropyconsistency} shows consistency with a single entropy condition.
\end{proof}

\subsection{Compensated compactness}\label{sec:compcomp}
We briefly summarize the technical compactness lemmas here, and refer to \cite{ChenI, TartarI} for more details.
\begin{lemma}[Murat's Lemma]\label{lm:Murat}
Let $\Omega \subset \R^d$, $d \geq 2$ be an open, bounded subset. Let $(\mu)_{n \in \N}$ be a bounded sequence in $W^{-1,p}(\Omega)$ for some $2 < p \leq \infty$. Suppose also that $\forall \; n \in \N$
\begin{equation}\label{eq:MuratDecomposition}
\mu_{n} = \xi_{n} + \pi_{n},
\end{equation}
where $\xi_{n}$ lies in a compact set of $H^{-1}_\loc(\Omega)$ and $\pi_{n}$ lies in a bounded set of $\mathcal{M}_\loc (\Omega)$. Then $(\mu_n)_{n \in \N}$ lies in a compact subset of $H^{-1}_\loc(\Omega)$. 
\end{lemma}

\begin{theorem}[Panov, Theorem 5 in \cite{PanovI}]\label{thm:panov}
Let $(u_\eps)_{\eps > 0}$ be a bounded sequence in $L^\infty(\R^d \times \R_+)$ such that for every $k \in \R$, the set
\begin{equation}\label{eq:Precompactness}
\bigl\{\partial_t \eta_k(u_\eps) + \nabla\cdot q_k(u_\eps)\bigr\}_{\eps > 0}
\end{equation}
is precompact in $H^{-1}_\loc(\R^d \times \R_+)$. (Here, $(\eta_k,q_k)$ denote the Kruzkov entropy pairs \eqref{eq:KruzkovEntropy}.) Then there is a subsequence $\eps_n \to 0$ as $n\to\infty$ and a function $u \in L^\infty(\R^d \times \R_+)$ such that
\begin{equation}
u_{\eps_n} \to u \quad \text{a.e.\ and in } L^p_\loc(\R^d \times \R_+) \text{ for every } 1 \leq p < \infty.
\end{equation}
\end{theorem} 

The following corollary shows that it is enough to consider \emph{smooth} entropies in the above result.
\begin{corollary}\label{cor:PanovI}
Let $(u_\eps)_{\eps > 0}$ be a bounded sequence in $L^\infty(\R^d \times \R_+)$ such that for every entropy pair $(\eta,q)$ with $\eta\in C^2_b(\R)$, the set
\begin{equation}\label{eq:PrecompactnessI}
\bigl\{\partial_t \eta(u_\eps) + \nabla\cdot q(u_\eps)\bigr\}_{\eps > 0}
\end{equation}
is precompact in $H^{-1}_\loc(\R^d \times \R_+)$.  Then there is a subsequence $\eps_n \to 0$ as $n\to\infty$ and a function $u \in L^\infty(\R^d \times \R_+)$ such that
\begin{equation}
u_{\eps_n} \to u \quad \text{a.e.\ and in } L^p_\loc(\R^d \times \R_+) \text{ for every } 1 \leq p < \infty.
\end{equation}
\end{corollary}
\begin{proof}
For $k\in\R$ and $m \in \N$, let $\eta_{k,m}\in C^2(\R)$ be convex functions satisfying
\begin{equation*}%\label{eq:SmoothEntropyApprox}
\begin{cases}
\|\eta_{k,m}'\|_{L^\infty(\R)}+\|\eta_{k,m}''\|_{L^\infty(\R)} <\infty \\
\|\eta_{k,m} - \eta_k\|_{L^\infty(\R)} \to 0 & \text{as } m\to\infty.
\end{cases}
\end{equation*}
Then the corresponding $C^2$ entropy flux $q_{k,m}(u)=\int_k^u f'(s)\eta_{k,m}'(s)\,ds$ also satisfies $\|q_{k,m} - q_k\|_{L^\infty(\R)} \to 0$ as $m \to \infty$. Fix $k\in\R$ and decompose
\begin{equation}\label{eq:entropydecomposition}
E_{\eps}:=\partial_t \eta_k(u_\eps) + \nabla\cdot q_k(u_\eps) = E_{\eps,m}^1 + E_{\eps,m}^2
\end{equation}
where 
\begin{align*}
E_{m,\eps}^1 &:=\partial_t \eta_{k,m}(u_\eps) + \nabla\cdot q_{k,m}(u_\eps), \\
E_{m,\eps}^2 &:= \partial_t\big(\eta_k(u_\eps)-\eta_{k,m}(u_\eps)\big) + \nabla\cdot\big(q_k(u_\eps)-q_{k,m}(u_\eps)\big).
\end{align*}
Let $(\eps_n)_{n\in\N}$ be a sequence converging to 0 as $n\to\infty$ and fix an open, bounded subset $U\subset\R^d\times\R_+$. It is enough to prove the existence of a further subsequence $E_{\eps'_n}$ which converges in $H^{-1}(U)$. For every $\phi\in H_0^1(U)$ we have
\begin{align*}
\ip{E_{m,\eps}^2}{\phi} &\leq \int_U \big|\eta_k(u_\eps)-\eta_{k,m}(u_\eps)\big\|\partial_t\phi| + \big|q_k(u_\eps)-q_{k,m}(u_\eps)\big\|\nabla\cdot\phi|\,d(x,t) \\
&\leq \underbrace{\big(\|\eta_k-\eta_{k,m}\|_{L^\infty(\R)} + \|q_k-q_{k,m}\|_{L^\infty(\R)}\big)|U|^{1/2}}_{\text{$\to0$ as $m\to\infty$, uniformly in $\eps>0$}}\|\phi\|_{H^1(U)}.
\end{align*}
Let $(\eps_n')_{n\in\N}$ be a subsequence of $(\eps_n)_{n\in\N}$, taken such that $(E_{m,\eps_n'}^1)_{n\in\N}$ is convergent in $H^{-1}(U)$ for every $m\in\N$. In particular, we may assume that $\|E_{m,\eps_n'}^1-E_{m,\eps_k'}^1\|_{H^{-1}(U)}\leq 1/m$ for every $k,n\geq m$. Then
\begin{align*}
\|E_{\eps_n'}-E_{\eps_k'}\|_{H^{-1}(U)} &\leq \|E_{m,\eps_n'}^1-E_{m,\eps_k'}^1\|_{H^{-1}(U)} + \|E_{m,\eps_n'}^2\|_{H^{-1}(U)}+\|E_{m,\eps_k'}^2\|_{H^{-1}(U)} \\
&\to 0
\end{align*}
as $n,k\geq m$ and $m\to\infty$.
\end{proof}

\subsection{Convergence of TECNO}\label{sec:convergenceproof}
The TECNO scheme (Definition \ref{def:tecno}) is guaranteed to dissipate the square entropy $\tilde{\eta}(u)=u^2/2$, but the discrete entropy residual \eqref{eq:SCLFVMentropy} might have either sign. We can nonetheless show that the entropy residual is not too large, in the following sense:

\begin{lemma}\label{lem:entropyresidualbound}
Assume that the solution computed by the TECNO scheme is $L^\infty$ bounded, \eqref{eq:linftybound}. Then for any entropy pair $(\eta,q)$ with $\eta\in C^2$, the total discrete entropy production is upper-bounded by
\begin{equation}\label{eq:TotalEntropyI}
\begin{split}
\bigl|\partial_t\eta(u^\DxDy) + \nabla\cdot Q\bigr|\bigl(\R^2\times[0,T]\bigr) \leq C \int_0^T \sum_{i,j} \Big(|\jmp{u}_{\iphf,j}|^3 \Dy + |\jmp{u}_{i,\jphf}|^3 \Dx \Big)\,dt \\
+ C \int_0^T \sum_{i,j} \Big(\jmp{u}_{\iphf,j}\jmpr{u}_{\iphf,j} \Dy + \jmp{\tilde{u}}_{i,\jphf}\jmpr{\tilde{u}}_{i,\jphf} \Dx \Big)\,dt
\end{split}
\end{equation}
where $\nabla\cdot Q$ denotes the measure whose integral of any $\phi\in C_c^0(\R^2\times\R_+)$ is
\begin{equation}\label{eq:discreteEntropyFluxDef}
\begin{split}
\ip{\nabla\cdot Q}{\phi}= \sum_{i,j\in\Z}\int_{\R_+}\overline{\phi}_{\iphf,j}\frac{Q_{\iphf,j}^x-Q_{\imhf,j}^x}{\Dx}+\overline{\phi}_{i,\jphf}\frac{Q_{i,\jphf}^y-Q_{i,\jmhf}^y}{\Dy}\,dt\Dx\Dy, \\
\overline{\phi}_{\iphf,j}:=\frac{1}{\Dy}\int_{y_\jmhf}^{y_\jphf}\phi(x_\iphf,y,t)\,dy, \qquad \overline{\phi}_{i,\jphf}:=\frac{1}{\Dx}\int_{x_\imhf}^{x_\iphf}\phi(x,y_\jphf,t)\,dx
\end{split}
\end{equation}
and where $Q^x,Q^y$ are given by \eqref{eq:entropyresidualflux}.
\end{lemma}

\begin{proof}
Let $\tilde{\psi}$ be the entropy potential with respect to the square entropy $\tilde{\eta}=u^2/2$, $\tilde{\psi}(u) = uf(u) - \tilde{q}(u)$. We split the entropy residual $r$ in \eqref{eq:entropyresidualflux} as $r=r^1+r^2$, where 
\[
r_{\iphf,j}^1 = \jmp{v}_{\iphf,j}\tilde{F}^x_{\iphf,j} - \jmp{\psi^{x}}_{\iphf,j}, \qquad 
r_{\iphf,j}^2 = -\jmp{v}_{\iphf,j}D^{x}_\iphf\jmpr{u}_{\iphf,j}
\]
and similarly for $r^{y}_{i,\jphf}$. The first part of the entropy residual can be estimated as
\begin{align*}
|r_{\iphf,j}^1| &\leq \Big| \jmp{v}_{\iphf,j}\tilde{F}^x_{\iphf,j} - \jmp{\psi^{x}}_{\iphf,j} \Big| \\ 
&= \bigg|\jmp{v}_{\iphf,j} \bigg( \frac{\jmp{\tilde{\psi^{x}}}_{\iphf,j}}{\jmp{u}_{\iphf,j}} - \frac{\jmp{\psi^{x}}_{\iphf,j}}{\jmp{v}_{\iphf,j}} \bigg) \bigg| \\
&= \big| \jmp{v}_{\iphf,j} \big| \bigg|\frac{1}{\jmp{u}_{\iphf,j}} \int_{u_{i,j}}^{u_{i+1,j}} \tilde{\psi}'(v)\,dv - \frac{1}{\jmp{v}_{\iphf,j}} \int_{v_{i,j}}^{v_{i+1,j}} \psi'(v) \,dv \bigg| \\ 
&= \big| \jmp{v}_{\iphf,j} \big| \bigg|\frac{1}{\jmp{u}_{\iphf,j}} \int_{u_{i,j}}^{u_{i+1,j}} f(u) \,du - \frac{1}{\jmp{v}_{\iphf,j}} \int_{v_{i,j}}^{v_{i+1,j}} f(u(v)) \,dv \bigg| \\ 
\intertext{\textit{(by the mean value theorem)}}
&= \big| \jmp{v}_{\iphf,j} \big| \bigg|\frac{f(u_{i,j}) + f(u_{i+1,j})}{2} - \frac{\jmp{u}^2_{\iphf,j}}{12} \psi'''(\tilde{\xi}_{\iphf,j}) \\ 
& \qquad - \frac{f(u_{i,j}) + f(u_{i+1,j})}{2} + \frac{\jmp{v}^2_{\iphf,j}}{12} \psi'''(\xi_{\iphf,j})\bigg| \\ 
\intertext{\textit{(by the $L^\infty$ bound on $u$)}}
&\leq C \big|\jmp{u}_{\iphf,j}\big|^3,
\end{align*}
and similarly in the $y$-direction,
\[
|r_{i,\jphf}^1| \leq C \big|\jmp{u}_{i,\jphf}\big|^3.
\]
The second part of the entropy residual can be bounded by
\[
|r_{\iphf,j}^2|\leq \|\eta''\|_{L^\infty([-M,M])}\overline{D}\jmp{u}_\iphf\jmpr{u}_\iphf.
\]
The conclusion now follows.
\end{proof}

We can now show precompactness of the sequence of approximations:

\begin{lemma}\label{lm:precompactness}
Let $\Omega \subset \R^2 \times [0,T]$ be a bounded subset and assume that the solution computed by the TECNO scheme is $L^\infty$ bounded, \eqref{eq:linftybound}. Then there is a subsequence $\DxDy'\to0$ such that $u^{\DxDy'}\to u$ pointwise a.e.~and in $L^p_\loc(\R^2 \times \R_+)$ for $1 \leq p < \infty$, for some $u\in L^1\cap L^\infty(\R^2\times\R_+)$.
\end{lemma}
\begin{proof}
Let $(\eta,q)$ be an arbitrary $C^2$ entropy pair. By Corollary \ref{cor:PanovI} it is sufficient to show that the sequence $\mathcal{E}_{\Dx,\Dy}:=\partial_t\eta(u^\DxDy) + \nabla\cdot q(u^\DxDy)$ is precompact in $H^{-1}_\loc$, and to this end we will employ Murat's lemma. Firstly note that $\mathcal{E}_{\Dx,\Dy}$ is bounded in $W^{-1,\infty}(\R^2\times\R_+)$, by the $L^\infty$ bound on $u^\DxDy$. Decompose
\begin{align*}
\partial_t\eta(u) + \nabla\cdot q(u) &= \underbrace{\partial_t\eta(u) + \nabla\cdot Q}_{=:\,\mathcal{E}^1} + \underbrace{\nabla\cdot q(u) - \nabla\cdot Q}_{=:\,\mathcal{E}^2}
\end{align*}
where $Q$ is given in \eqref{eq:entropyresidualflux}. Note that, due to the $L^\infty$ bound on $u^\DxDy$ and Theorem \ref{thm:tecnoproperties}, also $Q$ is locally Lipschitz continuous. By Lemma \ref{lem:entropyresidualbound} and Theorem \ref{thm:tecnoproperties}\textit{(iii)}, the discrete entropy production $\mathcal{E}^1$ is bounded in the space of measures $\mathcal{M}(\R^d\times\R_+)$.

Now to show that $\mathcal{E}^2$ is precompact in $H^{-1}_\loc(\R^2\times[0,T])$, let $\Omega\subset\R^2\times[0,T]$ be open and bounded and let $\phi\in H^1_0(\Omega)$. Extending $\phi$ by zero outside $\Omega$, we get
\begin{align*}
\mathcal{E}^2(\phi) &= \int_\Omega \phi\,d\big(\nabla\cdot q(u) - \nabla\cdot Q\big) \,d(x,y,t) \\ 
\intertext{(\textit{cf.~\eqref{eq:discreteEntropyFluxDef}})}
&= \int_0^T \sum_{i,j} \overline{\phi}_{\iphf,j}\big(q^x(u_{i+1,j})-q^x(u_{i,j})\big)\Dy + \overline{\phi}_{i,\jphf}\big(q^y(u_{i,j+1})-q^y(u_{i,j})\big)\Dx\,dt \\
&\quad-\int_0^T\sum_{i,j\in\Z}\overline{\phi}_{\iphf,j}\frac{Q_{\iphf,j}^x-Q_{\imhf,j}^x}{\Dx}+\overline{\phi}_{i,\jphf}\frac{Q_{i,\jphf}^y-Q_{i,\jmhf}^y}{\Dy}\,dt\Dx\Dy \\
\intertext{(\textit{summation by parts})}
&= \int_0^T \sum_{i,j} \frac{\overline{\phi}_{\iphf,j}-\overline{\phi}_{\imhf,j}}{\Dx}\big(Q_{\imhf,j}^x-q^x(u_{i,j})\big)\,dt\Dx\Dy \\
&\quad + \int_0^T \sum_{i,j} \frac{\overline{\phi}_{i,\jphf}-\overline{\phi}_{i,\jmhf}}{\Dy}\big(Q_{i,\jmhf}^y-q^y(u_{i,j})\big)\,dt\Dx\Dy \\
\intertext{(\textit{letting $\mathcal{I}=\{(i,j):\Omega\cap\cell_{i,j}\neq\emptyset\}$})}
&\leq \|\partial_x\phi\|_{L^2(\Omega)} \Bigg(\int_0^T \sum_{(i,j)\in\mathcal{I}} \big|Q_{\imhf,j}^x-q^x(u_{i,j})\big|^2 \Dx \Dy \,dt\Bigg)^{\frac{1}{2}} \\ 
& \quad + \|\partial_y\phi\|_{L^2(\Omega)} \Bigg(\int_0^T \sum_{(i,j)\in\mathcal{I}} \big|Q_{i,\jmhf}^y-q^y(u_{i,j})|^2 \Dx \Dy \,dt\Bigg)^{\frac{1}{2}} \\ 
\intertext{(\textit{by Lipschitz continuity of $Q$})}
&\leq C \|\phi\|_{H^1(\Omega)} \Bigg[\Bigg(\int_0^T \sum_{(i,j)\in\mathcal{I}}|\jmp{u}_{\iphf,j}|^2 \Dx \Dy \,dt\Bigg)^{\frac{1}{2}} + \Bigg(\int_0^T \sum_{(i,j)\in\mathcal{I}}|\jmp{u}_{i,\jphf}|^2 \Dx \Dy \,dt\Bigg)^{\frac{1}{2}} \Bigg] \\ 
&\leq C \|\phi\|_{H^1(\Omega)} |\Omega|^{\frac32} \Bigg[ \Bigg(\int_0^T \sum_{i,j}|\jmp{u}_{\iphf,j}|^3 \Dx \Dy \,dt\Bigg)^{\frac13} + \Bigg(\int_0^T \sum_{i,j}|\jmp{u}_{i,\jphf}|^3 \Dx \Dy \,dt\Bigg)^{\frac13} \Bigg] \\ 
& \to 0
\end{align*}
where the last step follows from \ref{eq:weakBVbound}. Thus by invoking Murat's Lemma \ref{lm:Murat} we can conclude that the sequence $(\mathcal{E}_{\Dx,\Dy})_{\Dx,\Dy>0}$ is precompact in $H^{-1}_\loc(\R^2 \times [0,T])$. Applying Corollary \ref{cor:PanovI} then yields the desired result.
\end{proof}

Now we need to show that this limit function $u$ is indeed a weak solution of \eqref{eq:SCLII}. To do so we state and prove the following ``Lax--Wendroff result''.

\begin{lemma}\label{lem:LaxWendroff}
Under the same assumptions as in Lemma \ref{lm:precompactness}, the limit $u$ is a weak solution of \eqref{eq:SCLII}.
\end{lemma}

\begin{proof}
Let $\phi \in C^1_c(\R^2 \times (0,T))$ be a test function and select a compact set $K_x \times K_y \subset \R^2$ such that $\supp\phi\subset K_x\times K_y\times[0,T]$. Furthermore, denote 
\[
\phi_{i,j}(t)=\phi(x_i,y_j,t), \qquad \phi^\DxDy(x,y,t) = \sum_{i,j} \phi_{i,j}(t) \chi_{\cell_{i,j}}(x,y).
\]
Multiplying the numerical scheme \eqref{eq:SCLFVMI} by $\phi_{i,j}(t)$ and integrating/summing in time/space, we get
\begin{align*}
& 0 = \int_0^T \Dx \Dy \sum_{i,j} \bigg(\phi_{i,j} \frac{d}{dt} u_{i,j}^\DxDy + \phi_{i,j} \frac{\tilde{F}^x_{\iphf,j} - \tilde{F}^x_{\imhf,j}}{\Dx} + \phi_{i,j} \frac{\tilde{F}^y_{i,\jphf} - \tilde{F}^y_{i,\jmhf}}{\Dy} \\ 
& \qquad - \phi_{i,j} \frac{D^{x}_{\iphf,j} \jmpr{u}_{\iphf,j} - D^{x}_{\imhf,j} \jmpr{u}_{\imhf,j}}{\Dx} \\ 
& \qquad - \phi_{i,j} \frac{D^{y}_{i,\jphf} \jmpr{u}_{i,\jphf} - D^{y}_{i,\jmhf} \jmpr{u}_{i,\jmhf}}{\Dy} \bigg) \,dt.
\end{align*}
After performing integration and summation by parts for temporal and spatial variables respectively we get
\begin{equation}\label{eq:LaxWendroffI}
A^1+A^2+A^3+ B^1+B^2 = 0
\end{equation}
where we can write
\begin{align*}
A^1 &:= -\int_0^T \Dx \Dy \sum_{i,j} u_{i,j} \frac{d}{dt} \phi_{i,j} \,dt\\
A^2 &:= -\int_0^T \Dx \Dy \sum_{i,j} \tilde{F}^x_{\iphf,j} \frac{\phi_{i+1,j} - \phi_{i,j}}{\Dx}\,dt \\
A^3 &:= -\int_0^T \Dx \Dy \sum_{i,j} \tilde{F}^y_{i,\jphf} \frac{\phi_{i,j+1} - \phi_{i,j}}{\Dy}\,dt \\ 
B^1 &:= \int_0^T\Dx\Dy\sum_{i,j} D^{x}_{\iphf,j} \jmpr{u}_{\iphf,j} \frac{\phi_{i+1,j} - \phi_{i,j}}{\Dx}\,dt \\
B^2 &:= \int_0^T\Dx\Dy\sum_{i,j} D^{y}_{i,\jphf} \jmpr{u}_{i,\jphf} \frac{\phi_{i,j+1} - \phi_{i,j}}{\Dy} \,dt.
\end{align*}
We can write $A^1=-\int_0^T\int_\R\int_\R u^\DxDy\partial_t\phi^\DxDy\,dx\,dy\,dt$, and thanks to the convergence of $u^\DxDy$ to $u$ from Lemma \ref{lm:precompactness} and convergence of $\phi^\DxDy$ to $\phi$ a.e., we have $\lim_{\Dx,\Dy\to0}A^1=- \int_0^T \int_{\R^2} u \partial_t\phi \,dx \,dy \,dt$. 

For the second term $A^2$, we denote for the sake of simplicity $\Delta_x \psi(x,y,t) = \frac{\psi(x+\Dx,y,t) - \psi(x,y,t)}{\Dx}$, for any function $\psi$. Since $\tilde{F}^x$ is a two-point flux, we can write
\begin{align*}\label{eq:LaxWendroffII}
A^2 &= -\int_0^T \int_{\R^2} \tilde{F}^x\big(u^\DxDy (x,y,t),\, u^\DxDy(x+\Dx,y,t)\big) \Delta_x \phi^\DxDy(x,y,t) \,d(x,y) \, dt \\ 
&= A^{2,1}+A^{2,2},
\end{align*}
where
\begin{align*}
A^{2,1} &:= -\int_0^T \int_{\R^2} f^x\big(u^\DxDy(x,y,t)\big) \Delta_x \phi^\DxDy(x,y,t) \,d(x,y) \,dt, \\ 
A^{2,2} &:= \int_0^T \int_{\R^2} \Big(f^x\big(u^\DxDy(x,y,t)\big) - \tilde{F}^x\big(u^\DxDy (x,y,t), u^\DxDy(x+\Dx,y,t)\big)\Big)\Delta_x \phi^\DxDy(x,y,t) \,d(x,y) \, dt.
\end{align*}
Thanks to the convergence of $u^\DxDy$ from Lemma \ref{lm:precompactness} and the a.e.~convergence of $\phi^\DxDy$ to $\phi$ we have
\begin{equation}\label{eq:LaxWendroffIII}
A^{2,1} \to - \int_0^T \int_{\R^2} f^x(u) \partial_x\phi \,d(x,y) \,dt \qquad \text{as } \DxDy\to0. 
\end{equation}
For the term $A^{2,2}$ we apply the H\"{o}lder inequality and Lemma \ref{lm:precompactness} to get
\begin{align*}
|A^{2,2}| &\leq \int_0^T \int_{\R^2} \Big|\tilde{f}^x\big(u^\DxDy (x,y,t)\big) - \tilde{F}^x\big(u^\DxDy(x,y,t), u^\DxDy(x+\Dx,y,t)\big)\Big| \big|\Delta_x\phi^\DxDy(x,y,t)\big| \,d(x,y) \, dt \\
\intertext{(\textit{using Lipschitz continuity of $\tilde{F}^x$})} 
& \leq C \int_0^T \int_{\R^2} \big|u^\DxDy (x+\Dx,y,t) - u^\DxDy(x,y,t)\big| \big|\Delta_x \phi^\DxDy(x,y,t)\big| \,d(x,y) \, dt \\ 
%& \leq C \int_0^T \Dx \Dy \sum_{\substack{x_i \in K_x \\ y_j \in K_y}} \Big|u_{i+1,j} - u_{i,j}\Big| |\Delta_x \phi| \,dt \\ 
%\intertext{(\textit{letting $\mathcal{I}=\{(i,j)\in\Z^2 : \cell_{i,j}\cap\supp\phi\neq\emptyset\}$})}
& \leq C \bigg(\int_0^T \int_{\R^2} \big|\Delta_x\phi^{\DxDy}\big|^{\frac32} \,d(x,y) \,dt \bigg)^{\frac23} \bigg(\int_0^T \sum_{i,j} \big|\jmp{u}_{\iphf,j}\big|^3 \Dx \Dy\,dt\bigg)^{\frac13} \\ 
& \leq C \|\partial_x \phi \|_{L^{\frac32}(\R^2 \times [0,T))} \bigg(\int_0^T \sum_{i,j} \big|\jmp{u}_{\iphf,j}\big|^3 \Dx \Dy\,dt\bigg)^{\frac13} \\ 
& \to 0
\end{align*}
as $\Dx, \Dy \to 0$ by \eqref{eq:weakBVboundu}. Analogously, $A^3 \to - \int_0^T \int_{\R^2} f^y(u) \partial_y \phi \,dx \,dy \,dt$ as $\Dx,\Dy\to0$. We conclude that 
\begin{equation}\label{eq:LaxWendroffIV}
A \to - \int_0^T \int_{\R^2} \Big[ u \partial_t\phi + f(u) \cdot \nabla \phi \Big] \,d(x,y) \,dt \qquad \text{as } \Dx, \Dy \to 0.
\end{equation}
It remains to show that $B^1,B^2$ in \eqref{eq:LaxWendroffI} vanish as $\Dx,\Dy\to0$. Indeed,
\begin{align*}
|B^1| &\leq \overline{D}\int_0^T\Dx\Dy\sum_{i,j} \big|\jmpr{u}_{\iphf,j}\big| \bigg|\frac{\phi_{i+1,j} - \phi_{i,j}}{\Dx}\bigg|\,dt \\
\intertext{(\textit{by \eqref{eq:ENOupperbound}})}
&\leq C\overline{D}\|\eta''\|_{L^\infty(\R)}\int_0^T\Dx\Dy\sum_{i,j} \big|\jmp{u}_{\iphf,j}\big| \bigg|\frac{\phi_{i+1,j} - \phi_{i,j}}{\Dx}\bigg|\,dt \\
&\leq C\overline{D}\|\eta''\|_{L^\infty(\R)}\|\partial_x\phi\|_{L^{3/2}(\R^2\times[0,T])} \left(\int_0^T\Dx\Dy\sum_{i,j} \big|\jmp{u}_{\iphf,j}\big|^3\,dt\right)^{1/3} \\
&\to 0
\end{align*}
as $\Dx,\Dy\to0$, and likewise for $B^2$. This completes the proof.
\end{proof}

Although we are not able to show that the TECNO scheme converges to the entropy solution, we will show that the weak solution $u$ satisfies at least one of the entropy conditions.
\begin{lemma}\label{lem:entropyconsistency}
With the same assumptions as in Lemma \ref{lm:precompactness}, the limit $u$ satisfies
\begin{equation}\label{eq:EntropyStableLW}
\partial_t\tilde{\eta}(u) + \nabla \cdot \tilde{q}(u) \leq 0.
\end{equation}
\end{lemma}
\begin{proof}
As in \eqref{eq:SCLFVMentropy} in Lemma \ref{lem:entropyresidualbound} we can write
\begin{equation}\label{eq:EntropyStableLWdisc}
\frac{d}{dt}\tilde{\eta}(u_{i,j}) + \frac{\tilde{Q}^x_{\iphf,j} - \tilde{Q}^x_{\imhf,j}}{\Dx} + \frac{\tilde{Q}^y_{i,\jphf} - \tilde{Q}^y_{i,\jmhf}}{\Dy} = \frac{r^{x}_{\iphf,j} + r^{x}_{\imhf,j}}{2\Dx} + \frac{r^{y}_{i,\jphf} + r^{y}_{i,\jmhf}}{2\Dy}
\end{equation}
where, in this particular case, the entropy residuals $r$ on the right-hand side are all nonpositive (see e.g.~\cite{TadmorI,UlrikII,UlrikIII}). Multiplying the above by a nonnegative test function $\phi \in C^1_c(\R^2 \times (0,T))$ and proceeding in the same manner as in Lemma \ref{lem:LaxWendroff} we obtain \eqref{eq:EntropyStableLW} in the sense of distribution.
\end{proof}

\section{Conclusions and outlook}
We prove convergence of the second-order TECNO scheme in two space dimensions to a weak solution of the hyperbolic conservation law \eqref{eq:SCLI}; this can easily be generalized to any number of space dimensions. The proof of this result relies on estimating the entropy residual appropriately using the (weak) TV bound obtained from entropy stability with respect to \textit{one} entropy. Invoking this estimate, precompactness of the sequence of approximate solutions is shown using a corollary derived from a \textit{compensated compactness} result due to Panov. Finally, to show that the limit function obtained due to the precompactness property is indeed a weak solution of \eqref{eq:SCLI}, a ``Lax--Wendroff'' type argument is used.

Convergence proofs of \textit{even higher-order} (i.e. more than second order) TECNO scheme in multiple space dimensions, to a weak solution of the equation \eqref{eq:SCLI} are still unanswered. In our opinion, this is largely due to the unavailability of \textit{weak} TV estimates of the type \eqref{eq:weakBVboundu}, as well as an appropriate version of Lemma \ref{lem:entropyresidualbound}. This should be an object of interest for future research. Last, but not least, one key estimate to prove \eqref{eq:weakBVboundv}, and consequently \eqref{eq:weakBVboundu}, is \eqref{eq:enoconjecture}. For \textit{even higher-order} ENO reconstruction, this estimate \eqref{eq:enoconjecture} (the ``ENO-conjecture'') is still not established and is still an open problem.

\bibliography{Reference}
\bibliographystyle{amsplain}
\end{document}